\documentclass[11pt]{amsart}
\usepackage[utf8]{inputenc}
\usepackage{amsthm,mathtools}
\usepackage{amsmath}
\usepackage{amsfonts}
\usepackage{stackrel}
\usepackage{graphicx}
\usepackage[english]{babel}

\usepackage[utf8]{inputenc}
\usepackage{subcaption}
\usepackage[export]{adjustbox}
\usepackage{mathabx}
\usepackage{wrapfig}
\usepackage{comment}
\usepackage[margin=1.2in]{geometry}
\usepackage{float}
\usepackage{tikz-cd,amssymb}
 \tikzcdset{every label/.append style = {font = \normalsize}}
\usepackage{adjustbox}
\usepackage{thmtools}
\usepackage{thm-restate}
\usepackage{tkz-graph}
\usepackage{tkz-berge}
\usetikzlibrary{decorations.pathreplacing,decorations.markings}
\usepackage{pinlabel}
\usepackage{soul}

\usepackage[dvipsnames]{xcolor}
\usepackage{hyperref}
\hypersetup{
    colorlinks=true,
    linkcolor=NavyBlue,
    citecolor=NavyBlue,
}

\usepackage[neverdecrease]{paralist}

\newtheorem{thm}{Theorem}[section]

\newtheorem{prop}[thm]{Proposition}
\newtheorem{lem}[thm]{Lemma}
\newtheorem{cor}[thm]{Corollary}

\newtheorem{conj}[thm]{Conjecture}
\newtheorem{Main}[thm]{Main Theorem}

\theoremstyle{definition}
\newtheorem{defn}[thm]{Definition}
\newtheorem{example}[thm]{Example}
\newtheorem{remk}[thm]{Remark}

\newcommand{\dpcross}{
\begin{tikzpicture}[baseline=-2.75,scale=.2]
\draw[-{Stealth[ length=1.25mm, width=1.25mm]},thick ](1,-1)--(-1,1);
\draw[-{Stealth[ length=1.25mm, width=1.25mm]},thick ](-1,-1)--(1,1);
\draw[fill=black] (0,0) circle (.2);
\end{tikzpicture}}

\newcommand{\smooth}{
\begin{tikzpicture}[baseline=-2.75,scale=.2]
\draw[-{Stealth[ length=1.25mm, width=1.25mm]},thick ] (-.8,-1) to[out=45,in=-45] (-1,1);
\draw[-{Stealth[ length=1.25mm, width=1.25mm]},thick ] (.8,-1) to[out=135,in=215] (1,1);
\end{tikzpicture}}

\addto\extrasenglish{%
}

\makeatletter
\let\c@conjecture=\c@theorem
\let\c@corollary=\c@theorem
\let\c@proposition=\c@theorem
\let\c@lemma=\c@theorem
\let\c@definition=\c@theorem
\let\c@problem=\c@theorem
\let\c@example=\c@theorem
\let\c@remark=\c@theorem
\let\c@equation\c@theorem
\let\c@question\c@theorem
\makeatother

\def\makeautorefname#1#2{\expandafter\def\csname#1autorefname\endcsname{#2}}
\makeautorefname{proposition}{Proposition}
\makeautorefname{lemma}{Lemma}
\makeautorefname{definition}{Definition}
\makeautorefname{problem}{Problem}
\makeautorefname{example}{Example}
\makeautorefname{question}{Question}
\makeautorefname{conjecture}{Conjecture}
\makeautorefname{section}{Section}
\makeautorefname{claim}{Claim}
\makeautorefname{equation}{Equation}
\makeautorefname{remark}{Remark}
\makeautorefname{corollary}{Corollary}

\def\R{\mathbb{R}}

\title{Danceability, A New Definition of Bridge Index}

\begin{document}
\author[Addison]{Sol Addison}
\address{Elon University}
\email{saddison3@elon.edu}

\author[Blackwell]{Sarah Blackwell}
\address{University of Virginia}
\email{blackwell@virginia.edu}
\urladdr{\url{https://seblackwell.com/}}

\author[Pongtanapaisan]{Puttipong Pongtanapaisan}
\address{Pitzer College}
\email{puttip@pitzer.edu}

\author[Scherich]{Nancy Scherich}
\address{Elon University}
\email{nscherich@elon.edu}
\urladdr{\url{https://nancyscherich.com/}}

\author[Snodgrass]{Lila Snodgrass}
\address{Elon University}
\email{lsnodgrass@elon.edu}

\author[Sullivan]{Everett Sullivan}
\address{Clayton State University}
\email{esullivan4@clayton.edu}

\begin{abstract}
    There are many commonly known definitions of the bridge index coming from combinatorial knot theory, Morse theory, geometry, and algebra.  In this paper, we prove that the danceability index is a new equivalent definition of the bridge index which can be seen as an oriented version of the Wirtinger number. We extend the danceability invariant to virtual knots in multiple ways and compare these invariants to two different notions of bridge index for virtual knots. 
\end{abstract}

\maketitle

\section{Introduction} \label{sec:introduction}

The bridge index is a well-known invariant first studied by Schubert in the 1950s \cite{HS}, which is a common tool for measuring the complexity of knots and links. For example, the set of 2-bridge knots is regularly used to test knot-theoretic hypotheses as it is a set of non-trivial, yet somewhat simple, knots. There are many equivalent definitions of the bridge index stemming from different mathematical motivations, but the three most classically well-known definitions come from  combinatorial knot theory, Morse theory, and geometry. 
More recently in 2017, Blair-Kjuchukova-Velazquez-Villanueva introduced an algebraic knot invariant called the Wirtinger index\footnote{The Wirtinger index was originally called Wirtinger \emph{number}, but we have chosen the convention to use \emph{number} to be the diagram-dependent value and \emph{index} to be the invariant value.} and proved it equivalent to the bridge index as well \cite{BKVV}.

The definition of bridge index used in this paper is the combinatorial version stated below.

\begin{defn}[\cite{HS}]\label{def:bridge}
    A \textbf{bridge} of a link diagram is a maximal length unbroken arc that includes at least one overcrossing. 
The \textbf{bridge number} of a link diagram is the number of bridges in the diagram.
The \textbf{bridge index} of a non-trivial link $K$, denoted $br(K)$, is the minimum bridge number taken over all diagrams representing $K$.  
The bridge index of the unknot is defined to be 1.
\end{defn}

\begin{thm}[\cite{schultenspaper,schultensbook,MIL,Rolfsen, BKVV}] \label{thm:bridge}
The bridge index is equal to the following link invariants.
\begin{enumerate}
    \item \emph{(Morse theory.)} The minimum number of local maxima taken over all Morse embeddings of a link. (For a nice exposition and proof, see  \cite{schultenspaper} or \cite[Section 4.8]{schultensbook}.)
    \item \emph{(Geometry.)} The \textbf{crookedness} of a knot\footnote{For the sake of simplicity, for this part of the theorem we only claim equivalence of the invariants in the case of knots.} (closely related to the total curvature), first defined by Milnor \cite{MIL}. Given  $\vec{v}(t)$ a unit circle parameterization of a knot, and  $\vec{w}$ a fixed vector in $\R^3$, one can count the relative maximums of $\vec{w}\cdot \vec{v}(t)$. The crookedness of a knot is the minimum such count taken over all choices of $\vec{w}$ and unit circle parameterizations of the knot; see \cite[Chapter 4.D.19]{Rolfsen}. 
    \item \emph{(Algebra.)} The \textbf{Wirtinger index} of a link is the minimum number of  colors needed to color a link while satisfying a Wirtinger relation at each crossings \cite{BKVV}. 
\end{enumerate}
\end{thm}


In 2021, inspired by the knotted paths of dancers on a stage, mathematician and dancer Karl Schaffer first introduced the notion of the danceability of a knot and questioned how many dancers are required to ``dance'' a given knot \cite{KS}. 
Shortly after in 2024,  Addison-Scherich-Snodgrass formalized Schaffer's ideas into an invariant called the danceability index and proved that this invariant is bounded above by the braid index of a knot \cite{ASS}.  In their paper, the authors conjectured that the danceability index is equal to the bridge index.
We defer the reader to \autoref{sec:DanceabilityIndex} for a complete definition of the danceability index, but informally, this invariant is the minimum number of dancers needed to traverse a link while adhering to strict rules at the crossings.
The main result of this paper is to prove that this notion is equivalent to the bridge index.

\begin{Main}
    The danceability index is equal to the bridge index.
\end{Main}



It can be readily shown that the bridge index is an upper bound for danceability, as we prove in \autoref{prop:upperbound}. 
For the reverse inequality, we give two different proofs. One proof uses a rudimentary diagrammatic technique called the \emph{bridge slide move}.
For the second proof, we turn to the Wirtinger index. As discussed in \autoref{sec:wirt}, the danceability index is closely related to the Wirtinger index, though the spirit and motivation for each invariant is quite different. The danceability index can be seen as an oriented version of the Wirtinger index with restrictions on which strands are considered ``adjacent" or not.
It is straightforward to see that the Wirtinger index is a lower bound for the danceability index, and so we recover our second proof from the result that the Wirtinger index is equal to the bridge index \cite{BKVV}.


We also discuss how the danceability index may be extended to virtual knots, though there is more than one way to make this extension; there are different choices for crossing rules at virtual crossings since there is no over or under-strand.
Interestingly, the bridge index also extends to virtual knots in more than one way; the virtual analogues of the bridge index from \autoref{def:bridge} and the Morse theory bridge index from \autoref{thm:bridge} are not equal for virtual knots \cite{NS}. 
In the final section of this paper, we explore comparisons between three virtual danceability invariants, the two different virtual bridge indexes, and several virtual generalizations of the Wirtinger index. 




\subsection{Organization of the paper} \autoref{sec:DanceabilityIndex} 
is dedicated to defining the danceability index. The danceability index, as originally defined in \cite{ASS}, required dancers to follow the \emph{under-first rule} at every crossing. In this paper, we give two different definitions of the danceability index, one using the \emph{under-first rule} and the other using the \emph{over-first rule}. We prove these two definitions induce the same invariant, and so in the following sections, we have the choice to adhere to either rule. At the end of \autoref{sec:DanceabilityIndex}, we briefly discuss computation of the danceability number of a diagram using discrete objects. \autoref{sec:dance=bridge} contains our first proof that the danceability index is equal to the bridge index. \autoref{sec:wirt} compares the danceability index to the Wirtinger index and gives our second proof that the danceability index is equal to the bridge index. Finally, \autoref{sec:virtual} concerns the extensions of the danceability index to virtual knots.

\subsection{Acknowledgments}
The authors would like to thank the anonymous referee from the journal \emph{Proceedings of Bridges 2024} who refereed paper \cite{ASS}. The result in \autoref{prop:upperbound} is due in part to the discussion with this referee. 
We would like to thank Marc Kegel for helpful conversations that also led to \autoref{prop:upperbound}.
We would also like to thank the anonymous referee from \textit{Topology and its Applications} for very helpful comments that lead to significant improvements the paper.

This project is part of an undergraduate student research experience at Elon University for Sol Addison and Lila Snodgrass, led by professor Nancy Scherich.
This project is funded in part by the Lumen Prize and the Summer Faculty Research Stipend at Elon University, and partially sponsored by the National Science Foundation under grant DMS-2532699. SB was supported by the NSF Postdoctoral Research Fellowship DMS-2303143.

\section{The danceability index}\label{sec:DanceabilityIndex}

As originally defined by Schaffer \cite{KS},   a knot diagram is \emph{danceable} if we can choose an orientation of the diagram and a point on that knot so that a dancer can start at that point, dance in the directional flow of the knot, and traverse the entire diagram with the restriction that the dancer must pass through every crossing as the under-strand first. 
This can be expanded to include multiple dancers, and multiple knotted components, so that if a dancer reaches a crossing on the over-strand having not already passed the crossing as the under-stand first, they must wait at the crossing for another dancer to ``clear the path" for the over-strand dancer. 
Visually, one can imagine the dancers' paths drawing the knot over time when viewed from above. With this definition, time is pointing upward, so later times must always travel on top of earlier times. 
This is explained below as the ``under-first rule" in \autoref{defn:configuration}.
However, if we assume time is pointing downward, then the dancers will trace the over-strand of a crossing first before crossing as the under-stand. This is called the ``over-first rule".

On a given oriented link diagram, the choice of under-first versus over-first can change the total number of dancers needed to traverse the entire diagram.
The danceability index was first defined using the under-first rule to agree with Schaffer's original motivation \cite{ASS}.
However, we prove that the induced danceability invariants for each rule are equal, so the choice of rule is irrelevant.
Since our goal is to prove that the danceability index is equivalent to the bridge index, it is easier to work with the over-first version of the definition since bridges are arcs with overcrossings.



\begin{defn}\label{defn:configuration}
    An \textbf{$n$-dance-under configuration} (resp. \textbf{$n$-dance-over configuration}) of an oriented link diagram is $n$ initial points on the diagram so that conditions (1), (2), and the respective third condition are satisfied.
    

\begin{enumerate}
    \item One dancer starts at each of the $n$ initial points, for a total of $n$ dancers.

    \item  Each dancer travels in the direction of the pre-chosen orientation of the diagram, and stops dancing when they reach the next initial point. This ensures that the total path traveled by all dancers traverses the entire diagram with no overlap.


    \item  The speed that each dancer travels can vary and can be chosen so that the simultaneous tracing of the diagram by all $n$ dancers follows the under-first rule (resp. over-first rule) at every crossing.
    \begin{itemize}
    \item[-] \underline{Under-first rule:} At every crossing, a dancer must pass through the crossing on the under-strand first before a dancer can cross on the over-strand. If a dancer reaches a crossing on an over-strand, they must wait until another dancer passes on the under-strand.

    \item[-] \underline{Over-first rule:} At every crossing, a dancer must pass through the crossing on the over-strand first before a dancer can cross on the under-strand. If a dancer reaches a crossing on an under-strand, they must wait until another dancer passes on the over-strand. 
    \end{itemize}
\end{enumerate}
\end{defn}

\begin{defn}
    For an oriented diagram $D$, the \textbf{under-first danceability number} (resp. \textbf{over-first danceability number}) is the minimum number $n$ so that there exists an $n$-dance-under configuration  (resp. an $n$-dance-over configuration) of $D$. This is denoted $da^u(D)$ (resp. $da^o(D)$). 

\end{defn}

\begin{figure}
\centering
\begin{picture}(150,80)
\put(-15,0){\includegraphics[scale=.4]{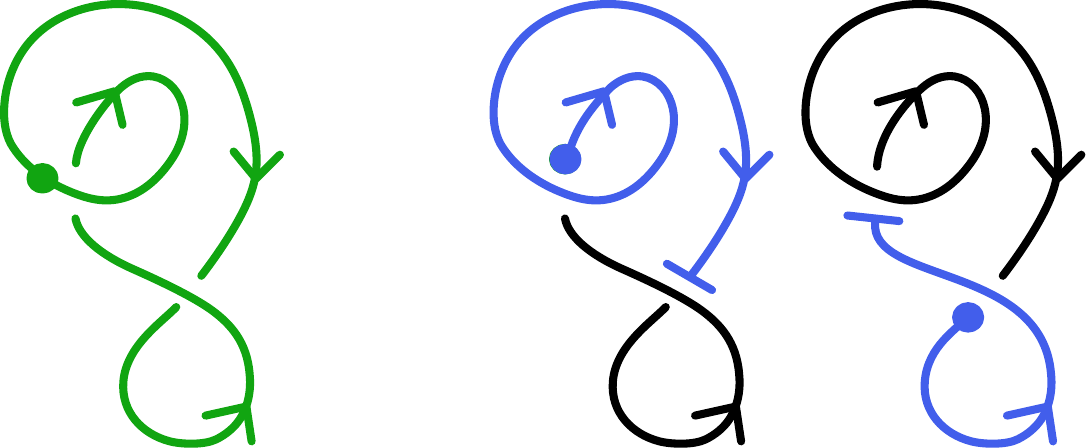}}
\put( -20,6){(A.)}
\put(73,6){(B.)}

\end{picture}
 \caption{The same oriented knot diagram requiring different numbers of dancers depending on whether the under-first or over-first rule is chosen. A dancer's starting point is designated by a dot, and a location where a dancer is waiting for another dancer to clear their path is designated by a $\dashv$. Diagram (A.) shows a 1-dance-under configuration for the knot diagram adhering to the under-first rule. The two figures in diagram (B.) show that this diagram cannot be danced with one dancer adhering to the over-first rule. (Here we only show the two possible staring points at the beginning of bridges in light of \autoref{lem:startatbridge}.)}
    \label{fig:comparison}
\end{figure}

The diagrams in \autoref{fig:comparison} illustrate  the same oriented knot diagram following first the under-first rule and then the over-first rule. We can see that the under-first danceability number is 1, while the over-first danceability number is 2. 
In this example, whether there is a positive or negative twist makes all the difference; because of the orientation, the dancer might be forced to enter the crossing as an under-strand versus the over-strand. 

The fact that simply adding a twist in the ``wrong" direction can  change the danceability of any given diagram demonstrates how the danceability of an oriented diagram using the under-first rule generally does not equal the danceability of that same oriented diagram using the over-first rule.  Each of the under-first and over-first danceability numbers induce a link invariant by taking the minimum over all possible diagrams for a link.

\begin{defn}
For an unoriented link $K$, the \textbf{dance-under index}, denoted $da^u(K)$, is the minimum under-first danceabilty number over all diagrams representing $K$, with any choice of orientation.
Analogously, the \textbf{dance-over index}, denoted $da^o(K)$, is the minimum over-first danceability number over all diagrams representing $K$, with any choice of orientation.
\end{defn}

A priori these two invariants seem different since diagrammatically, choice of orientation and over versus under can change the answer.
However, there is a relationship between these choices as outlined by the following retrograde trick and proposition, which will ultimately show that the induced invariants are, in fact, equal.

\vspace{2mm}\noindent \textbf{Retrograde trick:} Let $D^+$ and $D^-$ denote the two choices of oriented diagrams for an unoriented link diagram $D$. Suppose there is an $n$-dance-over configuration for $D^+$ with $n$ starting positions and varying speeds for the $n$ dancers to traverse the entire diagram in the positive orientation, while upholding the over-first rule. 
We can create an $n$-dance-under configuration for the oppositely oriented diagram $D^-$ by \emph{retrograding} the dance (i.e. watching the dance in reverse with respect to time). Choose the $n$ ending positions of the dance-over configuration to be the new starting points for the dance-under configuration. With time reversed, every crossing will uphold the under-first rule as the diagram is traversed in the opposite direction. This creates an $n$-dance-under configuration for $D^-$.
Analogously, taking the retrograde of a dance-under configuration will create a dance-over configuration with the same number of dancers.

\begin{prop}\label{prop:retrograde_on_diagram}
    For an unoriented link diagram $D$, the over-first danceability number of $D$ with one orientation is equal to the under-first danceability number of  $D$ with the opposite orientation. That is,
     let $D^+$ and $D^-$ denote the two choices of oriented diagrams for $D$. Then 
    \[da^o(D^+)=da^u(D^-).\] 
\end{prop}

\begin{proof}

There is a 1-1 correspondence between the $n$-dance configurations $D^+$ using the over-first rule and the $n$-dance configurations of $D^-$ using the under-first rule by the retrograde trick.
So, the minimal number of dancers required to dance $D^+$ with the over-first rule must also be the minimal number of dancers required to dance $D^-$ with the under-first rule.
\end{proof}

\begin{cor}
For any link K, the danceability invariant using the over-first rule is equal to the danceability invariant using the under-first rule, that is,
\[da^o(K)=da^u(K).\]

\end{cor}

\begin{proof}
    Suppose $da^o(K)=n$ and let $D^+$ be an oriented diagram  for $K$ realizing this minimum following the over-first rule.
    That is, $da^o(D^+)=n$. 
    Then by the retrograde trick, $da^u(D^-)=n$, where $D^-$ is an oriented diagram representing $K$.
    By minimality, we see $da^u(K)\leq da^u(D^-)= da^o(K)$.
    Reversing the argument, if $E^-$ is an oriented diagram for $K$ realizing the minimum following the under-first rule, $da^u(K)=da^u(E^-)$, then again by the retrograde trick we see that $da^o(K)\leq da^o(E^+)=da^u(K)$.
    Thus we can conclude that $da^o(K)=da^u(K).$
\end{proof}

Because the choice of under versus over-first does not affect the value of the induced invariant, we can define the danceability index as follows.

\begin{defn}\label{defn:danceability-1}
    For a link $K$, the \textbf{danceability index}, denoted $da(K)$, is $$da(K)\vcentcolon =da^o(K)=da^u(K).$$
    Moreover, the danceability index of $K$ is the minimum number of dancers needed to dance $K$ over all choices of oriented diagrams representing $K$, with the choice of either the over-first or under-first rule.
\end{defn}

\subsection{Computing danceability number from Gauss codes}

For a link diagram $D$, a dance configuration for $D$ can be described and stored in discrete objects through the use of Gauss codes. 
Since a dance configuration is only restricted by the timing of the dancers passing through undercrossings and overcrossings, the path of the dancer between crossings is not critical; we can discretize a dance configuration by supposing that as soon as a dancer passes through a crossing, they are immediately at the next crossing. So a Gauss code can capture the least amount of information needed to answer the problem of finding the minimum number of dancers for a diagram.

        \begin{defn}[\cite{Kauf}]\label{def:gausscodes}
			Let $D$ be an oriented link diagram with $n$ crossings. Let each crossing have a distinct label from $\{1,\ldots,n\}$, and let $n^{+}$ denote the overcrossing of crossing $n$ and $n^{-}$ the undercrossing of crossing $n$. 
            A \textbf{Gauss code} is a cyclic sequence of $n^{+}$'s and $n^{-}$'s for each component of $D$ obtained by traversing $D$ in the given orientation.\footnote{Gauss codes are also called ``signed Gauss sentences,'' or ``signed Gauss paragraphs'' for links; see \cite{KUR}.}
		\end{defn}
        \begin{figure}
            \centering
            \includegraphics[scale=1]{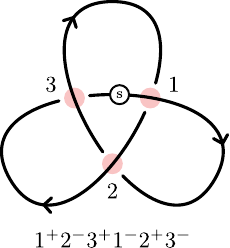}
            \caption{Example of Gauss code for the trefoil knot diagram.}
            \label{fig:Gauss_trefoil}
        \end{figure}
        
			\autoref{fig:Gauss_trefoil} shows an example Gauss code for the oriented trefoil diagram with labeled crossings. While the Gauss code is cyclic, a circled \emph{s} is indicated on the knot to show where the sequence started.

The location of dancers on a link diagram $D$ can be stored in a (cyclic) table whose columns are labeled by a chosen Gauss code for $D$. Since the only movement restrictions are at crossings, dancers can freely dance along arcs of the link that are not near a crossing. So the table will only record the location of a dancer on an arc right before the next crossing in their progression. For example, in \autoref{fig:example_progression} dancer $A$ is in the column labeled by $1^+$, which means dancer $A$ is on the over-strand of the crossing labeled 1 and has not yet passed over this crossing. 

As a dancer moves through the link diagram, the location of the dancer advances through the table to the next column. For example, in \autoref{fig:example_progression} dancer $B$ moved through crossing 3 on the over strand and is heading to crossing 1 on the under-strand. This change is documented in the table by advancing $B$ forward one entry--and this new configuration is stored as a new row of the table.
We will call this a ``dance progression,'' that is, the advancement of one dancer across the next crossing (either over or under) in front of them.

\begin{figure}
    \centering
    \includegraphics[scale=1]{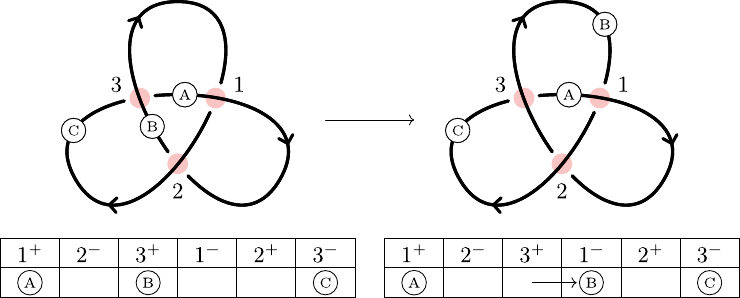}
    \caption{An example dance move where dancer $B$ advances forward through crossing 3.}
    \label{fig:example_progression}
\end{figure}

        \begin{figure}
            \centering
            \includegraphics[scale=1]{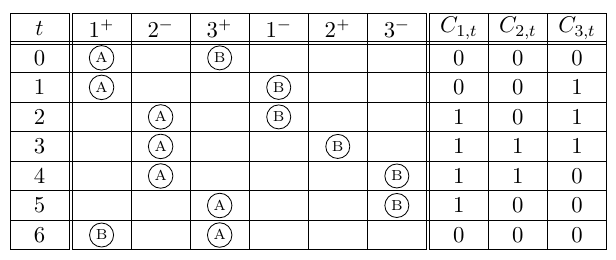}
            \caption{An example free-form 2-dance for the trefoil knot diagram which defines a 2-dance configuration for the diagram. The crossing-statuses are shown at the right of the table for each time increment. }
            \label{fig:free_forem_trefoil}
        \end{figure}

        \begin{defn}
			A \textbf{free-form $n$-dance} of an oriented link diagram $D$ is a table whose columns are labeled by a Gauss code for $D$, with rows labeled by time, such that each row shows the locations of the $n$ dancers along the diagram, and each row is one application of a dance progression applied to the previous row.
            \end{defn}

	\autoref{fig:free_forem_trefoil} shows a free-form $2$-dance for the trefoil diagram. To determine if a free-form $n$-dance gives a valid $n$-dance configuration for $D$, we need to check if the over-first rule is adhered to at each crossing. To do this, we can calculate the ``crossing-status'' of each crossing, as in the next definition.	
			\begin{defn}The \textbf{crossing-status of a crossing $n$ at time $t$}, denoted $C_{n,t}$, is the number of times a dancer has advanced through its overcrossing subtracted by the number of times a dancer has advanced through its undercrossing in the first $t$ dance progressions of the dance.\end{defn}

If a dancer crosses through a crossing as the under-strand first before the over-strand has been danced, the crossing status for that crossing will temporarily be negative.
            In conclusion, a free-form $n$-dance gives an $n$-dance configuration if the following two conditions are true:
            \begin{enumerate}
            \item \underline{Satisfies over-first rule:} The crossing-status of every crossing is non-negative for all time.
            \item \underline{The entire diagram is danced:} The last row and the first row have dancers in the same locations, but each individual dancer has shifted forward to finish in the starting position of the dancer ahead of them in the row (cyclically).
            \end{enumerate}

    Since these two conditions can be checked in polynomial time, verifying if a free-form $n$-dance is an $n$-dance configuration can be done in polynomial time of the number of crossings $O(n^{k})$ by the following algorithm:

        \begin{enumerate}
            \item
                Give an ordering to the dancers.
            \item 
                Select the first dancer to be the active dancer.
            \item 
                Check if the active dancer can move. If they can \emph{and they are not at the starting location of the next dancer}, do so, update crossing status, and go to step 4.
                Otherwise move to step 5.
            \item
                Check if all the dancers have advanced to the starting position of the next dancer. If so, the diagram has been danced, return true. (Since this is only checked after at least one dancer has moved, we are not testing the starting location.)
            \item 
                If the active dancer is the last dancer, then no dancer has been able to move, meaning the link is in deadlock, return false.
                Otherwise set the active dancer to be the next dancer and go to step 3.
        \end{enumerate}

        The algorithm loops through the list of dancers moving each one around the link.
        Since each dancer will make no more than $2n$ moves to complete the track, and supposing there are $d$ dancers, the  total number of moves is no more than $2 n\cdot d$.
        In the worse case only one move is made per loop. meaning that the number of times that a dancer is checked is a maximum of $2n\cdot d\cdot d = 2n\cdot d^{2}$.
        Since there are $2n$ spaces on the Gauss code and only one dancer is allowed per spot, we know that $d<n$ and so the algorithm returns an answer in $O(n^{3})$.

        In order to find the minimum number of dancers needed to complete a dance of a diagram, all possible starting positions must be considered. Since we only need to consider dancers starting at the beginning of a strand (by \autoref{lem:startatbridge}), there are $2^{n}$ possible starting configurations with a specified direction.
        Thus, considering both directions, there are $2\cdot2^{n}$ starting positions.

        For each starting position, we may run the algorithm to check if it is a valid $n$-dance configuration and return the smallest number needed.
        Thus the computation complexity of finding the fewest needed for a given link diagram is $O(2\cdot2^{n}\cdot n^{3}) = O(2^{n})$.\footnote{We can alternatively view this in terms of scheduling theory, viewing the Gauss code as a cyclic schedule of $n$ events.
        If each event has two sub-events, one of which must be done before the other, then determining the minimum number of dancers for a configuration is equivalent to determining the minimum number of parallel processes needed to run the schedule endlessly without deadlock.}





\section{The danceability index is equal to the bridge index}\label{sec:dance=bridge}

\begin{prop}\label{prop:upperbound}
The bridge index is an upper bound for danceability index: $da(K)\leq br(K)$.
\end{prop}

\begin{proof}
Let $D$ be a diagram of a link $K$ realizing the bridge index of $K$, that is, $br(D)=br(K)=n$.
On every bridge, place a dancer's initial starting point. 
A bridge is an unbroken arc made up of only overcrossings, so for a dancer following the over-first rule, each dancer can advance forward across all bridges at the same time. 
After this, all crossings have been crossed over as the over-strand first. 
Hence all dancers can continue to dance all arcs which are not bridges until the entire diagram is traversed. 
This shows that $D$ admits an $n$-dance-over configuration, and so $da(K)\leq da(D)\leq br(D)=br(K)$.
\end{proof}

 




\begin{lem}\label{lem:startatbridge}

A diagram with an $n$-dance-over configuration admits an $\ell$-dance-over configuration, with $\ell\leq n$, where every initial point is at the start of a bridge. 
\end{lem}

\begin{proof}

Suppose a diagram $D$ has an $n$-dance-over configuration.
    First note that every dancer's initial point is the ending point for the dancer dancing behind them in the orientation.

    Adhering to the over-first rule, a dancer can freely dance forward along a bridge. If two dancers have their initial points on the same bridge, one dancer can be removed and the remaining dancer's initial point can be dragged backwards to the start of the bridge.

    If a dancer has an initial point that is not on a bridge, we will show this dancer's starting point can be slid forward to the start of the next bridge.
    Suppose dancer $A$ has a starting point, denoted $a$ on an arc of $D$ that is between two underpasses, as shown in \autoref{fig:dragstartingpoint}. 
    Let $s$ denote the start of the first bridge in the path ahead of $a$.
    Let dancer $B$ denote the dancer behind dancer $A$.
    We need to show that dancer $A$ can start at $s$ and dancer $B$ can dance their entire path and then dance the path from $a$ to $s$ that $A$ is no longer dancing.
    
    Keeping the rest of the configuration the same, suppose dancer $A$ starts at $s$ instead of $a$. We progress through the diagram in \autoref{fig:dragstartingpoint} over time and discuss which strands have to be danced by whom and in what order.

    \setdefaultleftmargin{2cm}{2cm}{}{}{}{}
    \begin{enumerate}
    
    \item[First:] Since we started with an $n$-dance-over configuration, the overcrossings of underpass 2 must be danced first (and by dancers other than $A$) before $A$ could have advanced from $a$ to $s$.
    \item[Second:] Then $A$ is free to dance over the bridge starting at $s$. 
    \item[Third:] Both $A$ and $B$ are free to dance any part of the diagram until dancer $B$ can freely arrive at the undercrossings of underpass 1. 
    \item[Fourth:] Since we started with an $n$-dance-over configuration, all overcrossings of underpass 1 can be freely crossed by dancers other than $B$, but possibly including $A$. 
    \item[Fifth:] Dancer $B$ passes under underpass 1 and arrives at $a$, completing the original path for dancer $B$. 
    \item[Sixth:] Since the overcrossing of underpass 2 was crossed earlier in the dance, dancer $B$ is free to step forward through to $s$, completing the new configuration. 
    \end{enumerate}
\end{proof}

\begin{cor}
If a diagram $D$ has danceability number $n$, then there is at most one dancer's starting point on every bridge of $D$.
\end{cor}
\begin{figure}
			    \centering
			    
				\begin{tikzpicture}

				\draw[white,double=black,ultra thick,double distance=1.5pt, line cap=round] (9,1) -- (9,-1);

					\draw[-{Stealth[length=5mm]}, white,double=black,ultra thick,double distance=1.5pt, line cap=round] (0,0) -- (10,0);
     
                        \draw[-{Stealth[length=5mm]}, ultra thick, line cap=round] (9.5,0) -- (10,0);
					
					\draw[white,double=black,ultra thick,double distance=1.5pt, line cap=round] (2,1) -- (2,-1);
                         \node[] at (2.5,.75){\Large  $\cdot\cdot\cdot$};
                        \node[] at (2.5,-.75){\Large  $\cdot\cdot\cdot$};
					\draw[white,double=black,ultra thick,double distance=1.5pt, line cap=round] (3,1) -- (3,-1);
					
					\draw[white,double=black,ultra thick,double distance=1.5pt, line cap=round] (5,1) -- (5,-1);
                        \node[] at (5.5,.75){\Large  $\cdot\cdot\cdot$};
                        \node[] at (5.5,-.75){\Large  $\cdot\cdot\cdot$};
					\draw[white,double=black,ultra thick,double distance=1.5pt, line cap=round] (6,1) -- (6,-1);

				    \node[fill=white,shape=rectangle,draw,inner sep=2.3pt] at (1,0) {\small $B$};
                    \node[fill=white,shape=circle,draw,inner sep=2.3pt] at (4,0) {\small $a$};
                        \node[fill=white,shape=circle,draw,inner sep=2.3pt] at (6.5,0) {\small $s$};

                    \draw [decorate,decoration={brace,amplitude=5pt},xshift=0pt,yshift=0pt] (1.8,1) -- (3.2,1) node [black,midway,yshift=.5cm] {underpass 1};
					\draw [decorate,decoration={brace,amplitude=5pt},xshift=0pt,yshift=0pt] (4.8,1) -- (6.2,1) node [black,midway,yshift=.5cm] {underpass 2};
                    \draw[->](6.5,.5)--(6.5,.3);
                    \draw [] (7.5,.7)  node {\small start of bridge};

				\end{tikzpicture}
    \caption{}\label{fig:dragstartingpoint}
    \end{figure}

We now have what we need to prove the main result of the paper: the danceability index is equal to the bridge index. 
Resulting from \autoref{prop:upperbound}, all that remains to be shown is that the danceability index is an upper bound for the bridge index. 
The key ingredient for the proof is a bridge slide move, shown in \autoref{fig:path_and_slide}. This move is similar to the bridge reduction move used by Hirasawa-Kamada-Kamada to compare the bridge index definitions for virtual links \cite[Lemma 4.1]{HKK}. 
Our proof, although it does not rely on virtual crossings, uses the same idea of reducing the bridge number by applying a bridge slide move. 

\begin{thm}\label{thm:bridgeEquivDance}
     The bridge index is equal to the danceability index; $da(K) = br(K)$ for any classical link $K$.
\end{thm}

\begin{proof}
		

Since the bridge index is an upper bound for the danceability index, as proved in \autoref{prop:upperbound}, it suffices to show that the reverse inequality is true. That is, we will show that for any link $K$, $br(K)\leq da(K)$.
For this proof, we use the over-first rule to compute the danceability index.

Let $K$ be a link with $da(K)=k$ and suppose $D$ is an oriented diagram for $K$ that admits a $k$-dance-over configuration.
By \autoref{lem:startatbridge}, we can assume that the dancers' initial points are all at the start of bridges.
If every bridge of $D$ has exactly one dancer's initial starting point on it, then for the diagram $D$, $br(D)=da(D)$ which then implies that $br(K)\leq da(K)$. 
So, suppose $br(D)\geq da(D)$. Then, there is a bridge of $D$ without a dancer's initial point on it, and there must be a dancer, denoted dancer $A$,  whose path traverses at least two bridges. Let $a$ denote the initial point of dancer $A$.
The path of dancer $A$ starts just after an underpass (at the start of a bridge), followed by a bridge with at least one overpass, then an underpass, then a second bridge, and an underpass (and possibly more), as shown in \autoref{fig:path_and_slide} (A.).

\begin{figure}
                \centering
                \includegraphics[scale=1]{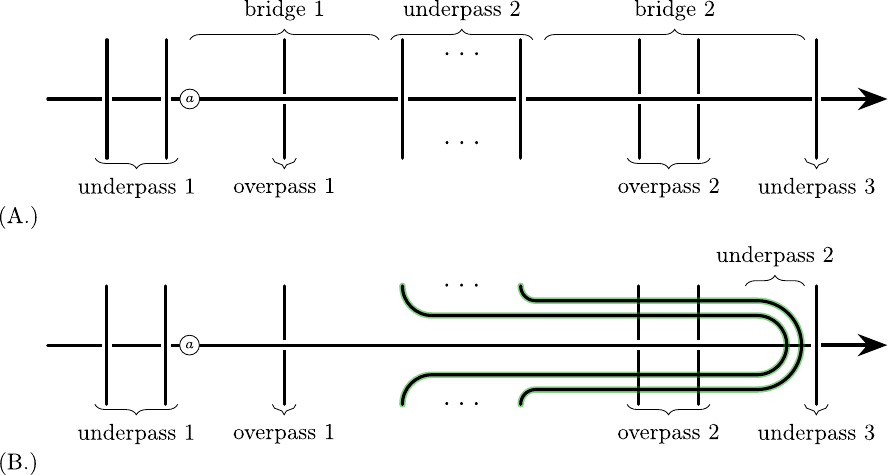}
			    \caption{(A.) A sketch of the path of dancer $A$ across two bridges. (B.) A bridge slide move applied to underpass 2. This is a local sequence of Reidemeister II and III moves to slide underpass 2 over overpass 2.}\label{fig:path_and_slide}
			\end{figure}

   Following \autoref{fig:path_and_slide} (A.), since the $k$-dance-over configuration for $D$ uses the over-first rule, dancer $A$ can freely dance along the first bridge and cross over overpass 1, but then must wait to pass under underpass 2 until other dancers have traveled over those crossings. 
Once all crossings in underpass 2 have been cleared, dancer $A$ can freely pass under underpass 2 and then over the next bridge over the overpass 2.
This implies that none of the crossings in underpass 2 can occur as a crossing in overpass 2 (otherwise dancer $A$ would have to traverse a crossing in underpass 2 as the under-stand first in order to pass over the crossing later, which violates the over-first rule).

Because the crossings in underpass 2 are distinct from the crossings in overpass 2, we can apply a bridge slide move to underpass 2, as shown in \autoref{fig:path_and_slide} (B.). 
A bridge slide is a sequence of Reidemeister II and III moves applied to slide all of the over-strands in underpass 2 past the under-strands of overpass 2 (along the second bridge).
After the bridge slide has been applied, call this new diagram $D^{\ast}$.
In $D^{\ast}$, underpass 2 is adjacent to underpass 3, which eliminates bridge 2; this bridge slide move has joined bridge 1 and bridge 2 into a single bridge and created a new diagram with bridge number reduced by 1.

We now show that the danceability of $D^{\ast}$ has not changed, that is $da(D^{\ast}) = da(D)=k$.
The bridge slide move introduced new crossings between the over strands of underpass 2 and the under strands of overpass 2.
For the dancers on the over-strands of underpass 2, the new crossings are all overcrossings for these dancers, and so they can advance without issue.
Note that in $D$ before the bridge slide was applied, every over-strand of underpass 2 must be danced before dancer $A$ can cross under and then move on to cross over overpass 2. 
All of the dancers on the under-strands of overpass 2 must wait for dancer $A$, who must wait for the over-strand dancers of underpass 2.
So in $D^{\ast}$, all of the dancers on the under-strands of overpass 2 can wait while the over-strand dancers of underpass 2 cross over all new crossings, maintaining danceability with the same original dancers as before the bridge slide was applied.

In conclusion, for every bridge in $D$ that does not have a dancer's initial starting point on it, the bridge slide move can be applied to eliminate that bridge while maintaining the danceability number. 
After all such bridges have been eliminated, we are left with a new diagram $D'$ in which every dancer dancers over exactly one bridge. 
Thus, $br(K)\leq br(D')= da(D')=da(K)$ which concludes the proof.
		\end{proof}




\section{Comparing danceability to the Wirtinger index} \label{sec:wirt}

Blair-Kjuchukova-Velazquez-Villanueva developed a link invariant called the Wirtinger index as the minimum number of ``seed" meridional generators needed to generate the fundamental group of the link compliment with a Wirtinger presentation \cite{BKVV}. 
They describe a coloring procedure to start with initial colored ``seed" strands and generate the link group via iterated application of the Wirtinger relation. We summarize their method in the following definition.



\begin{defn}[\cite{BKVV}]
\label{def:Wirt}
We refer to the arcs in a link diagram as \textbf{strands}. A link diagram $D$ is called \textbf{$k$-partially colored}, where $k$ is the number of starting colors, if there is a specified subset $A$ of the strands of $D$ and a function $f \colon A \rightarrow \{1, 2,\ldots, k\}$. The tuple $(A, f)$ is called a \textbf{$k$-partial coloring}. Given k-partial colorings $(A_1, f_1)$ and $(A_2, f_2)$ of $D$, we say $(A_2, f_2)$ is the result of a \textbf{coloring move} on $(A_1, f_1)$ if
\begin{enumerate}
    \item $A_1 \subset A_2$ and $A_2 \backslash  A_1 = {s_j}$ for some strand $s_j$ in $D$;
    \item $f_2|A_1 = f_1$;
    \item $s_j$ is adjacent to $s_i$ at some crossing $c \in v(D)$, and $s_i \in A_1$;
    \item the over-strand $s_k$ at $c$ is an element of $A_1$;
    \item $f_1(s_i) = f_2(s_j)$.
\end{enumerate}
A diagram $D$ is \textbf{fully colored} if every strand of the diagram has received a color.
The minimum value of $k$ such that $D$ can be fully colored following the above rules is known as the \textbf{Wirtinger number} of the diagram, denoted $\omega(D)$. The \textbf{Wirtinger index} of a link $K$ in $S^3$, denoted $\omega(K)$, is the minimal value of $\omega(D)$ over all diagrams $D$ for $K$.
\end{defn}

At a glance, both the danceability and Wirtinger indices involve coloring (or labeling) the arcs of a link diagram, and both have to satisfy rules at the crossings. 
For the Wirtinger index, two strands $s_i$ and $s_j$ are adjacent if they are the under-strands of a crossing. 
The coloring move described in  \autoref{def:Wirt} suggests that we can color adjacent under-strands the same color if the over-strand at that crossing $s_k$ and one of the two adjacent under-strands $s_i$ or $s_j$ is already colored. 
This coloring move resembles the over-first rule for danceability, that a dancer on the under-strand is able to ``color'' (continue dancing forward) the adjacent under-strand if the over-strand has already been ``colored'' (danced over). 

The main difference is that danceability imposes a restriction on the order in which adjacent under-strands can be colored. When computing the Wirtinger number on an unoriented diagram, it does not matter which under-strand $s_i$ or $s_j$ gets colored first. However, when computing the danceability number, there is a specific order in which $s_i$ and $s_j$ must be colored, as determined by the oriented flow assigned to the diagram which governs the directional movement of the dancers over time. 
As such, the danceability index can be seen as an oriented Wirtinger index.

Due to this additional orientation restriction, we can see that the Wirtinger index is bounded above by the danceability index.

\begin{prop}\label{prop:wirtatmostdance}
    $\omega(K)\leq da(K)$.
\end{prop}

\begin{proof}
The over-first rule is an oriented restriction of the adjacency coloring rule described in \autoref{def:Wirt}. Starting with an oriented link  diagram that realizes an $n$-dance-over configuration with $n$ initial starting points, pick the strands with a dancer's initial point to be the seed strands for the $n$-partial coloring. The timing of the progression of dance can be adjusted so that only one dancer passes through a crossing at any given time. Since the dancers must satisfy the over-first rule, this gives an indexing of the strands which describes a sequence of valid coloring moves of the diagram which result in a completed Wirtinger coloring with $n$ starting seeds.
\end{proof}

The argument in the proof of \autoref{prop:wirtatmostdance} cannot be simply reversed;
 given a Wirtinger coloring of a diagram, it is possible that the seed strands do not correspond to enough initial dancer starting points for a valid dance configuration. For example, \autoref{fig:wirt_comparison} (A.-D.) shows that the $5_2$ knot can be completely colored following the rules in \autoref{def:Wirt} with two seed strands. However, as shown in \autoref{fig:wirt_comparison} (E.-F.), no matter the choice of orientation, there does not exists a 2-dance configuration for the knot \emph{with starting points on the initial seeding strands} for the Wirtinger coloring.
 This example makes it all the more surprising that as link invariants, the Wirtinger index and the danceability index are equal.

 In fact, using the idea in \autoref{fig:wirt_comparison}, one can construct a link diagram $D$ such that $da(D)-\omega(D)$ can be arbitrarily large.

 \begin{figure}[h]
    \centering
    \begin{picture}(300,155)
       \put(0,7){ \includegraphics[width=0.75\linewidth]
       {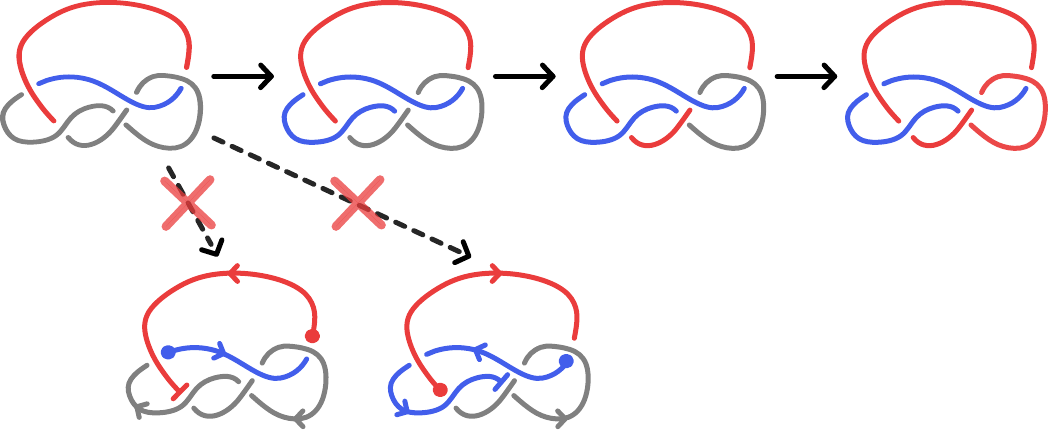} }
       \put(28,150){(A.)}
       \put(117,150){(B.)}
       \put(206,150){(C.)}
       \put(295,150){(D.)}
       \put(68,-4){(E.)}
       \put(150,-4){(F.)}
    \end{picture}
    \caption{Diagram (A.) shows a partially colored knot diagram with seed strands colored blue and red. Following the arrows to the right shows a sequence of valid coloring moves resulting in diagram (D.), a fully colored diagram with Wirtinger number 2. Diagrams (E.) and (F.) are oppositely oriented diagrams with initial dancer starting positions on the two seed strands colored in diagram (A.). Both diagrams (E.) and (F.) require a third dancer to complete the dance.}
    \label{fig:wirt_comparison}
\end{figure}

 \begin{prop}
     For every positive integer $n$, there exist a link $L$ and a diagram $D$ of $L$ with the property that $da(D)-\omega(D)\geq n.$\label{prop:differencebetweendanceandwirt}
 \end{prop}
\begin{proof}
    \autoref{fig:wirt_comparison} belongs to a larger family called ``twist knots.'' Consider the minimum crossing diagram $D'$ of the 6-crossing twist knot with two local maxima with respect to the standard height function on the plane. This assumption guarantees that $\omega(D')=2$ as the strands at the local maxima can function as seeds. The reader can check by hand, or by using our Python code found in \cite{table}, that $da(D') = 3$. That is, the entire diagram cannot be traversed by putting two dancers on any two strands initially. To construct a diagram satisfying the claim, we take $D$ to be the split union of $n$ copies of $D'.$ We see that $\omega(D)=\omega(\sqcup^n D')=2n$. On the other hand, $da(D)=da(\sqcup^n D')=3n$.
\end{proof}

\begin{cor}\label{cor:w=d}
    $br(K)=\omega(K)= da(K)$.
\end{cor}
\begin{proof}
    We showed in \autoref{thm:bridgeEquivDance} that the danceability index is equal to the bridge index, and Blair-Kjuchukova-Velazquez-Villanueva proved that the Wirtinger index is also equal to the bridge index \cite{BKVV}.
\end{proof}



It would be interesting if one could prove \autoref{cor:w=d} without appealing to bridge index.


\section{Danceability index for virtual links }\label{sec:virtual}
Virtual knot theory was first developed by Louis Kauffman in the 1990s \cite{Kauf}.
Virtual links are faithfully described by planar link diagrams with two types of crossings, the usual classical crossings and virtual crossings, up to the virtual Reidemeister moves. (The virtual Reidemeister moves can be found in \cite[Figure 1]{BK}.) Virtual crossings are denoted with a circle around the crossing; see \autoref{fig:4.53_knot} and \autoref{fig:virtual_trefoil} for a examples of virtual knot diagrams. One motivation for virtual knot theory comes from Gauss codes. From \autoref{def:gausscodes}, it is clear that every link diagram admits a Gauss code. However, not all Gauss codes can be realized as planar link diagrams--this is where a virtual crossing is introduced. 

The danceability index can be extended to virtual links in multiple ways. A virtual crossing is topologically and algebraically different from a classical crossing, and this difference may be reflected in the movement of the dancers at a virtual crossing. For example, at a virtual crossing, the dancers' paths may meet in a duet or a physical interaction such as a lift. This would inspire a different crossing rule for virtual crossings than for classical crossings. In the following definition, we give three different options for crossing rules for virtual crossings, which give rise to three danceability invariants for virtual links.



\begin{defn}

An  \textbf{unrestricted $n$-dance configuration} (resp. \textbf{coincident} or \textbf{smoothing $n$-dance configuration}) for an oriented virtual link diagram is $n$ initial points on the diagram satisfying conditions (1), (2), and (3), while following the over-first rule\footnote{Here we choose the over-first rule, but one could also choose the under-first rule. } at every classical crossing, as stated in \autoref{defn:configuration}, as well as the respective fourth condition at any virtual crossings.

\begin{enumerate}

    \item[(4)] The speed that each dancer travels can vary, and can be chosen so that the simultaneous tracing of the diagram by all $n$ dancers follows the respective rule.
    \begin{itemize}
    \item[-] \underline {Unrestricted rule:} At every virtual crossing, a dancer can freely pass through the crossing without any time restriction or consideration of the other dancers' paths.

    \item[-] \underline {Coincident rule:} At every virtual crossing, two dancers must pass through the crossing at the same time. If a dancer reaches a virtual crossing alone, they must wait until another dancer is available to pass with them simultaneously. 

    \item[-] \underline{Smoothing rule:}
    At every virtual crossing, choose an $\varepsilon$-radius neighborhood of the crossing that does not intersect the link diagram outside of the virtual crossing. Two dancers must enter and exit the neighborhood at the same time\footnote{One could omit this time restriction and simply smooth the virtual crossings, creating a dance configuration for the resulting classical link diagram. However, from the perspective of a live dance performance, the audience would not be able to distinguish a dance of a smoothed virtual link from a classical link. So, we chose to add the time restrictions at the virtual crossings to demonstrate that the dancer is at a virtual crossing.}. In the neighborhood, the dancers follow a local smoothing of the crossing, as shown in the diagram in \autoref{fig:smoothing_rule}. If a dancer reaches a neighborhood alone, they must wait until another dancer reaches the neighborhood to pass with them simultaneously.
    \end{itemize}
\end{enumerate}
\end{defn}

\begin{figure}[h]
    \centering
    \includegraphics[scale=.55]{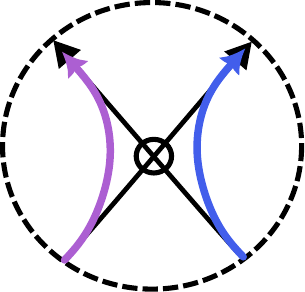}
    \caption{Two dancers' paths following the smoothing rule at a virtual crossing.}
    \label{fig:smoothing_rule}
\end{figure}

Following the same procedure outlined in \autoref{sec:DanceabilityIndex} for classical links, these three types of dance configurations induce three virtual link invariants.

\begin{defn}
    For an oriented virtual link diagram $D$, the \textbf{unrestricted number}, denoted $da^{(u)}(D)$, is the minimum number $n$ so that there exists an unrestricted $n$-dance configuration of $D$. The \textbf{coincident number} and \textbf{smoothing number} are defined analogously, and denoted $da^{(c)}(D)$ and $da^{(s)}(D)$ respectively.
\end{defn}

\begin{defn}\label{def:virtualda}
    For a virtual link $K$, the \textbf{unrestricted danceability index}, denoted $da^{(u)}(K)$, is the minimum unrestricted number over all diagrams representing $K$, with any choice of orientation. 
    The \textbf{coincident danceability index} and \textbf{smoothing danceability index} are defined analogously, and denoted $da^{(c)}(K)$ and $da^{(s)}(K)$ respectively.
\end{defn}

In \autoref{fig:4.53_knot}, we give an example of a virtual knot diagram with the three different danceability configurations.
The diagram represents the virtual 4.53 knot, and we denote it by $D$. \autoref{fig:4.53_knot} (A.) shows an unrestricted dance configuration for $D$ with two dancers, \autoref{fig:4.53_knot} (B.) shows a coincident dance configuration for $D$ with three dancers, and \autoref{fig:4.53_knot} (C.) shows a smoothing dance configuration for $D$ with three dancers. By simple trial and error, one can check that these are the minimal number of dancers required to dance this diagram, so $da^{(u)}(D) = 2$, $da^{(c)}(D) = 3$, and $da^{(s)}(D) = 3$.  We can see that depending on which definition of danceability is used, the minimum number of dancers required to traverse a given virtual diagram may differ. 

\begin{figure}[t]
\centering
\begin{picture}(350,120)
\put(0,0){\includegraphics[width=0.8\linewidth]{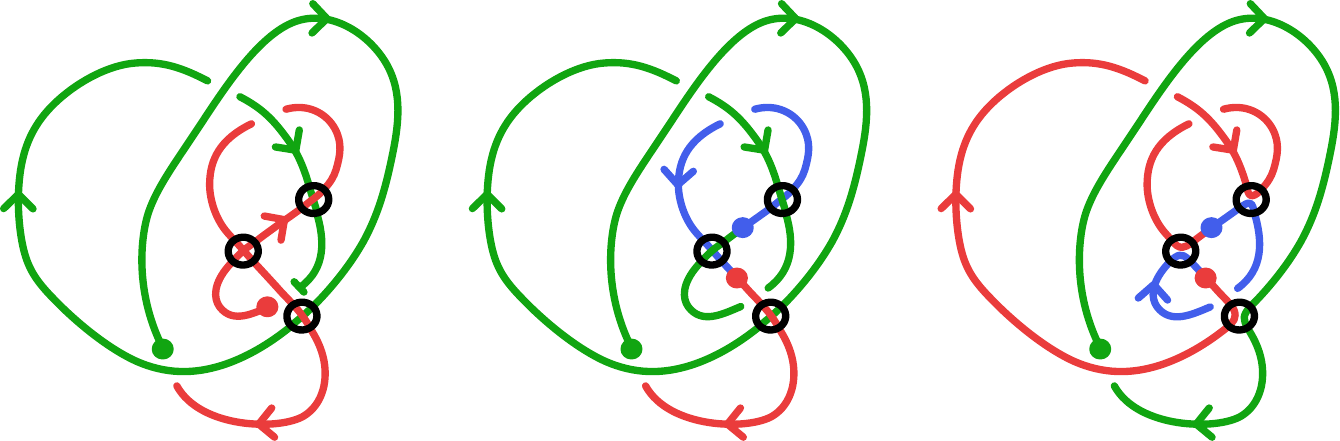}}
\put( 15,6){(A.)}
\put(140,6){(B.)}
\put(265,6){(C.)}
\end{picture}
\caption{The virtual 4.53 knot diagram with (A.) an unrestricted 2-dance configuration, (B.) a coincident 3-dance configuration, and (C.) a smoothing 3-dance configuration.}
    \label{fig:4.53_knot}
\end{figure}

Notice that in an unrestricted dance configuration, dancers can meet at a virtual crossing simultaneously and do a duet or pass through at different times, or even a single dancer can pass through a virtual crossing as both strands at separate times.
Thus, by definition, every coincident $n$-dance configuration is also an unrestricted $n$-dance configuration.
\begin{cor}\label{cor:ulessc}
    For all classical and virtual links $K$, $da^{(u)}(K)\leq da^{(c)}(K)$.
\end{cor}

We can also see there is a relationship between the coincident danceability index and the smoothing danceability index. 
There is a 1-1 correspondence between the coincident $n$-dance configurations and the smoothing $n$-dance configurations of any diagram. The idea is that at virtual crossings, dancers can ``trade roles" mid-dance and complete the configuration (following the different virtual crossing rule) by dancing someone else's path. 
The result of this is that the coincident danceability index is equal to the smoothing danceability index.

\begin{prop}
    For all classical and virtual links $K$, $da^{(c)}(K)=da^{(s)}(K)$.
\end{prop}

\begin{proof}
    For any link diagram $D$ and for every $n$, there is a 1-1 correspondence between the coincident $n$-dance configurations and the smoothing $n$-dance configurations of $D$. 
    To see this, fix a link diagram $D$ and a positive integer $n$.
    Given a coincident $n$-dance configuration for $D$, construct a smoothing $n$-dance configuration by following this procedure at every virtual crossing:
    \setdefaultleftmargin{.8cm}{}{}{}{}{}
    \begin{enumerate}
     \item   Choose a small $\varepsilon$-radius neighborhood that does not intersect $D$ outside of the crossing.
    \item Since we started with a valid coincident $n$-dance configuration, there are two dancers that pass through the virtual crossing at the same time: call them dancers $A$ and $B$. By altering their speeds slightly, the dancers can be made to enter the $\varepsilon$ neighborhood of the crossing at the same time. In the neighborhood, exchange the dancers' paths from $\dpcross$  to $\smooth$. (Instead of the dancers passing through the virtual crossing at the same time,  they ``bounce" off each other and swap positions.) Dancer $A$ leaves the neighborhood in dancer $B$'s place, and vice versa.
    \item Dancer $A$ dances the remainder of the configuration in dancer $B$'s place and dancer $B$ dances the remainder of the configuration in dancer $A$'s place. 
    
     \end{enumerate}
    Since the overall paths that are danced, and the times in which they are danced, are not affected outside small neighborhoods of the virtual crossings (although \emph{who} is dancing the paths has changed a lot), the resulting configuration is a valid smoothing $n$-dance configuration. 
    If the starting diagram $D$ did not have virtual crossings, then the coincident $n$-dance configuration is vacuously a smoothing $n$-dance configuration.
    Since the smoothing rule has time constraints on when two dancers must enter and exit a neighborhood of a crossing, it is clear to see that the  procedure described above is reversible, showing the desired 1-1 correspondence.

As a result, for any diagram $D$, the minimal coincident number is equal to the minimal smoothing number, which implies the coincident danceability index is equal to the smoothing danceability index for every link.
\end{proof}

As stated in the two previous proofs, if a diagram has no virtual crossings, then it is vacuously true that an unrestricted $n$-dance configuration is also a coincident $n$-dance configuration. It is also clear that if a diagram has no virtual crossings, then an unrestricted (or coincident) $n$-dance configuration is also just a (classical) $n$-dance-over configuration for $D$. This leads to the following, vacuously true, corollary.

\begin{cor}
    If a diagram $D$ has no virtual crossings, then $da(D)=da^{(u)}(D)=da^{(c)}(D)$.
\end{cor}

It is natural to question if these equalities extend to the danceability invariants when applied to classical links: for a classical link $K$, does $da(K)=da^{(u)}(K)=da^{(c)}(K)$?
This question centers around whether these invariants are minimized on a classical diagram; is it possible to decrease the invariants by adding virtual crossings. We will address this further in \autoref{sec:bridge_dance_virtual}.



The next natural question to ask is which, if any, of the virtual danceability indices equal the bridge number for virtual links. This turns out to be a more subtle question than in the classical case because there are different, non-equivalent, definitions of the bridge index for virtual links. Before we address this question, we discuss the various definitions of the bridge index for virtual links.

\subsection{Bridge index(es) and Wirtinger index for virtual links}

We begin with the following definition, which extends the notion of a bridge to the virtual setting.

\begin{defn}
    An \textbf{over-bridge} of a virtual link diagram is a maximal length unbroken arc of the diagram which includes at least one classical overcrossing. An arc can pass through a virtual crossing and remain unbroken, so an over-bridge is an arc of the diagram containing only classical overcrossings, and potentially some virtual crossings as well.
\end{defn}

Now we recall two definitions of the bridge index for virtual links, corresponding to two versions of the bridge index for classical links.

\begin{defn}[\cite{NS}]
The first definition of the \textbf{bridge index} for virtual links, denoted $b_1$, is given by the minimum number of over-bridges across all possible diagrams of a virtual link.
The second definition of the \textbf{bridge index} for virtual links, denoted $b_2$, is the minimum number of local maxima taken over all planar diagrams that are Morse with respect to a height function. 
\end{defn}

For classical links, the definition of $b_1$ agrees with the usual (combinatorial) definition of the bridge index, and the definition of $b_2$ agrees with the Morse-theoretic definition of the bridge index.
As mentioned in \autoref{thm:bridge}, for classical links these two definitions are always equal; however for virtual links this is not the case. 

Nakanishi-Satoh proved that if a virtual link $K$ has $b_2(K)$ equal to 1, then $K$ must be trivial \cite{NS}. (This aligns with the classical bridge index result that all non-trivial classical links must have bridge index at least 2.)
However, it is possible for a non-trivial virtual link to have $b_1$ equal to 1. 
For example, the 2.1 virtual trefoil, denoted $T_v$ and shown in \autoref{fig:virtual_trefoil}, has $b_1(T_v)=1$, and since $T_v$ is non-trivial, $b_1(T_v)=1<b_2(T_v)$.
In general, Nakanishi-Satoh proved that $b_1(K) \leq b_2(K)$ for all virtual links $K$, and this inequality can be strict \cite{NS}. 
Later, the authors of \cite{NPV} showed that the difference between $b_1(K)$ and $b_2(K)$ can be arbitrarily large. A stronger result stating that for any positive integers $m \leq n$, there is a virtual link $K$ with
$b_1(K) = m$ and $b_2(K) = n$, was also proved in \cite{KPW}.

\begin{figure}
    \centering
    \includegraphics[scale=0.5]{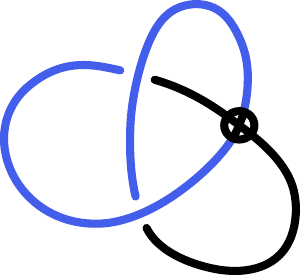}
    \caption{A Morse diagram for the 2.1 virtual trefoil knot, $T_v$, which achieves the minimal virtual bridge indexes $b_1(T_v) = 1$ and $b_2(T_v)=2$. 
    }
    \label{fig:virtual_trefoil}
\end{figure}

In \cite{PP}, one of the authors defined the Wirtinger index $\omega(K)$ for virtual links, and proved that $\omega(K)=b_1(K)$, using the characterization of virtual links as links in thickened surfaces. The definition of $\omega(K)$ is the same as \autoref{def:Wirt} except that a strand may include virtual crossings. Put differently, a strand can be extracted from a Gauss code as a string of positive integers, which
may be empty, appearing between two negative integers.



\subsection{Bridge and danceability index(es) for virtual links} \label{sec:bridge_dance_virtual}

For virtual links, which danceability index equals which bridge index? One relationship is clear to see; the unrestricted danceability index is less than or equal to the $b_1$ bridge index. Since in both definitions, the virtual crossings offer no constraints (an over-bridge arc is not broken at a virtual crossing, and an unrestricted dancer can freely pass through a virtual crossing at any time), the same proof as \autoref{prop:upperbound} applies in this virtual context.

\begin{prop}\label{prop:danceatmostbridge}
    For any virtual link $K$, $da^{(u)}(K)\leq b_1(K)$.
\end{prop}
\begin{proof}
Let $D$ be a diagram of the virtual link $K$ realizing the $b_1$ bridge index of $K$, that is, $b_1(D)=b_1(K)=n$.
On every over-bridge, place a dancer's initial starting point. 
Each dancer can advance forward across all over-bridges at the same time, satisfying the over-first rule at every classical crossing, and freely passing through any virtual crossings along the way. 
After this, all classical crossings have been crossed over as the over-strand first. 
Hence all dancers can continue to dance all arcs that are not over-bridges until the entire diagram is traversed. 
This shows that $D$ admits an unrestricted $n$-dance-over configuration, and so $da^{(u)}(K)\leq da^{(u)}(D)\leq b_1(D)=b_1(K)$.
\end{proof}

\begin{thm}\label{thm:da=b}
    For any virtual link $K$,  $da^{(u)}(K)\geq \omega(K)$ and  $da^{(u)}(K)=b_1(K)$.
\end{thm}

\begin{proof}
 The proof is essentially the same as \autoref{prop:wirtatmostdance}, but now we use the definition of strands for virtual links. That is, we start with an oriented link diagram that realizes an $n$-dance-over configuration with $n$ initial starting points. Pick the strands with a dancer's initial point to be the seed strands. It follows that $D$ can be fully colored with $n$ seed strands.

 Since $\omega(K)=b_1(K)$ as proved in \cite{PP}, and $b_1(K)\leq da^{(u)}(K)$ by \autoref{prop:danceatmostbridge}, then it follows that $da^{(u)}(K)=b_1(K)$.
\end{proof}

\begin{cor} 
    For any classical link $K$,  $da(K)=da^{(u)}(K).$
\end{cor}

\begin{proof}
    By \cite{BR}, the first bridge index does not decrease when $K$ is considered as a virtual link, so $br(K)=b_1(K)$ and therefore $da(K)=b_1(K)$.  By \autoref{thm:da=b}, $b_1(K)=da^{(u)}(K)$. 
\end{proof}

The question of whether a classical invariant can be decreased in a virtual link diagram is usually answered using parity projection techniques, as was done in \cite{BR} for the bridge index. We suspect that one could apply similar parity projection techniques to prove $da^{(c)}$ is also minimized for classical links on classical diagrams to prove the following conjecture.

\begin{conj}
    For any classical link $K$, $da(K)=da^{(c)}(K)$.
\end{conj}

\begin{remk} For classical links, we proved that the danceability index is equal to bridge index without appealing to Wirtinger index. The proof of \autoref{thm:bridgeEquivDance} does not extend to all virtual links because the bridge slide move might require the Over-Crossings-Commute (OCC) relation (also called ``Welded Reidemeister move'' or ``forbidden move''). 
To illustrate this, in \autoref{fig:virtual_path_and_slide} we show the bridge slide move in bridge 2 with specific crossings chosen to be virtual. 
When performing a bridge slide, any strand with a virtual crossing in underpass 2 can be slid to the right over both the classical and virtual crossings of overpass 2 by using the virtual Reidemeister II and III moves (the strand highlighted in green in \autoref{fig:virtual_path_and_slide}). Sliding this green strand creates virtual crossings on every strand in overpass 2.
On the other hand, any strand with a classical crossing in underpass 2 can be slid to the right \emph{only} over classical crossings in overpass 2 (the strand highlighted in red in \autoref{fig:virtual_path_and_slide}). To slide the red strand over a virtual crossing requires OCC. The quotient of the set of virtual links by the OCC relation is called the set of ``welded links.'' So, the bridge slide relation is allowable on welded links, and the proof of \autoref{thm:bridgeEquivDance} can only be used to prove the danceability index equals $b_1$ for welded links, not for all virtual links.

\begin{figure}[h]
    \centering
    \includegraphics[scale=1]{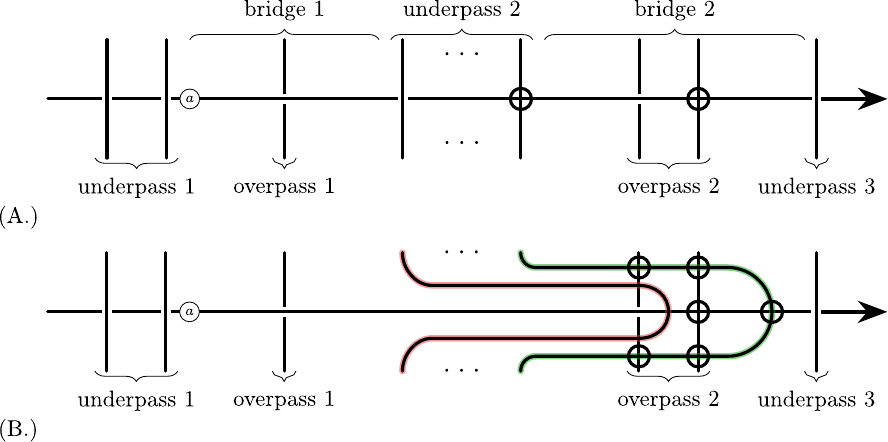}
    \caption{A bridge slide move applied to underpass 2 for virtual links with problematic red strand. }\label{fig:virtual_path_and_slide}
\end{figure}

\end{remk}

Now we turn to $b_2$: how does $b_2$ compare to these danceability indices? The short answer is that we do not know yet, but we can define unoriented versions of danceability that will be related to $b_2.$ In \cite{NPV}, the authors strove to find an analogue of Wirtinger index for $b_2(K)$ in the virtual setting. They introduced the notion of semi-Wirtinger number\footnote{In their paper, the authors used ``number'' to refer to the virtual link invariant, not just the diagram invariant, as is our convention.} $sw(K)$. While $sw(K)$ gives a better estimate for $b_2(K)$ than $\omega(K)$ for virtual links, it is unknown whether $b_2(K)=sw(K).$ 

Based on the discussions in this paper, we see that the coincident danceability index can be thought of as an ``oriented" semi-Wirtinger number, where the coincident rule is applied to classical crossings instead of virtual crossings. In fact, there is a unified way to compare the versions of Wirtinger indices, danceability indices, and quandle-theoretic invariants summarized in the following table. We leave some entries unnamed as an exercise for the readers to come up with the names for the invariants.

\vspace{1em}
\begin{center}
    \begin{tabular}{|c|c|c|}
\hline
    \textbf{Danceabilities} (oriented Wirtinger) & \textbf{Wirtinger indices} & \textbf{Quandle theories} \\
\hline
unrestricted $d^{(u)}$ & $\omega$ & quandles \\
\hline
coincident $d^{(c)}$ & $sw^{(v)}$ & virtual quandles \\
\hline
unnamed & $sw$ & biquandles \\
\hline
unnamed & unnamed & virtual biquandles \\
\hline 
\end{tabular}
\end{center}
\vspace{1em}

Here, the invariant $sw^{(v)}$ is defined in the exact same way as $sw$, except that a strand can include overcrossings, but not virtual crossings. The following statement holds by how we designed the table above. 

\begin{prop}
  $sw^{(v)}(K)\leq da^{(c)}(K)$.
\end{prop}\begin{proof}
    We start with a valid coincident $n$-dance configuration. The coincident rule is an oriented restriction of the adjacency coloring rule for the virtual semi-Wirtinger number. Turn every semi-arc (which in this case, is a portion of the link diagram containing only overcrossings) containing the dancer into a seed strand. The timing of the progression of dance can be adjusted so
that only one dancer passes through a classical crossing and two dancers pass through a virtual crossing at any given time. Since the dancers must
satisfy the coincident rule, this gives an indexing of the strands which describes a sequence
of valid coloring moves of the diagram which result in a completed coloring with $n$ starting
seed strands.
\end{proof}

The table above opens doors for algebraic lower bounds for all these versions of danceability. For definitions of the quandles involved, readers can consult \cite{ceniceros2009virtual,nelson2016quotient}. Essentially, one can define a presentation for any entry in the third column of the table above from a diagram by first defining what a strand is. That is, one decides whether a strand is an arc that is allowed to include over and virtual crossings (quandles), overcrossings only (virtual quandles), virtual crossings only (biquandles), or no crossings at all (virtual biquandles). Then, when strands interact at crossings, we get a relation in the presentation. 

Since the second row from the table is more appropriate for this section, we demonstrate the lower bound with virtual quandles, but the readers can work out the analogous result for any row of the table. Therefore, an alternative way to define the coincident danceability index is the minimum number of generators for the fundamental virtual quandles coming from propagation of certain initial strands of the diagrams by the virtual quandle rules.

\begin{defn}
A \textbf{quandle} is a set $Q$ equipped with an operation $\triangleright \colon Q\times Q\to Q$ 
satisfying for all $x,y,z\in Q$:
\begin{itemize}
\item[(i)] $x\triangleright x=x$,
\item[(ii)] The map $f_y \colon Q\rightarrow Q$ defined by $f_y(x)=x\triangleright y$ is a bijection, and
\item[(iii)] $(x\triangleright y)\triangleright z=(x\triangleright z)\triangleright(y\triangleright z)$.
\end{itemize}
\end{defn}

\begin{defn}
Let $Q$ be a quandle and $v \colon Q\to Q$ a quandle automorphism.
We say that $(Q,v)$ is a \textbf{virtual quandle} if $v$ satisfies
\[v(x\triangleright y)=v(x)\triangleright v(y).\]
\end{defn}

\begin{example}
We state an example originally given in  \cite{farinati2019virtual}. The virtual quandle $(S,i_{(2,3)})$ on the set $\{1,2,3\}$ defined by $x\triangleright y = 2y-x \pmod 3$ and $i_{(2,3)}(1)=1,i_{(2,3)}(2)=3,$ and $i_{(2,3)}(3)=2.$
\end{example}

\begin{defn}
A \textbf{virtual quandle coloring} of a virtual link diagram $D$ is an assignment of a virtual quandle element to each strand of $D$ such that the relations in \autoref{fig:virtquan} are satisfied at all classical and virtual crossings.
\end{defn}

\vspace{.5em}
\begin{figure}[h]
\labellist
\small\hair 2pt
\pinlabel $x\triangleright y$ at -4 138
\pinlabel $y$ at -12 6
\pinlabel $x$ at 142 6
\pinlabel $y$ at 272 6
\pinlabel $x$ at 426 6
\pinlabel $v(y)$ at 277 143
\pinlabel $v^{-1}(x)$ at 419 143
\endlabellist
    \centering    \includegraphics[scale=.4]{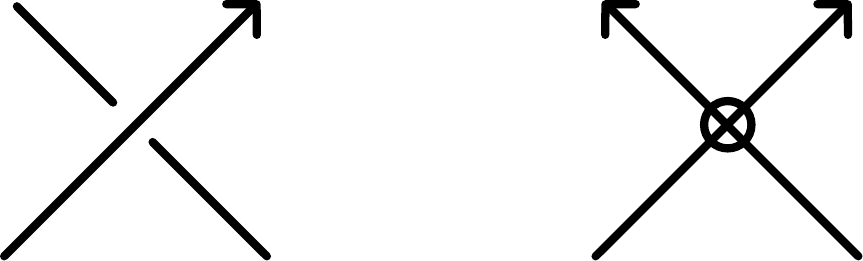}
    \caption{A label of strands of a diagram is a virtual quandle coloring if these relations are satisfied at all classical and virtual crossings}\label{fig:virtquan}
\end{figure}

\begin{example}
    The Kishino link in \autoref{fig:dancekish} (left) admits nine colorings by $(S,i_{(2,3)})$. In this figure, $a$ and $b$ can be any element of $\{1,2,3\}$, but the middle two strands are always colored with element 1.
\end{example}

We show that the coincident danceability index can be arbitrarily large without appealing to the lower bound from $da^{(u)}$. We accomplish this by showing that a virtual link can admit no non-trivial quandle colorings, but can admit an arbitrarily high number of non-trivial virtual quandle colorings. Let $Col_V(K)$ denote the number of colorings by a virtual quandle. It is well known that virtual Reidemeister moves do not affect $Col_V(K)$, and hence, $Col_V(K)$ is a virtual link invariant \cite{ceniceros2009virtual}.

\begin{lem}
    For any virtual link $K$, $\log_{|V|} Col_V(K) \leq da^{(c)}(K)$.\label{lem:relatevirtquanwithdance}
\end{lem}

\begin{proof}
    Using the coincident rule, we can place $da^{(c)}(K)$ dancers on the strands of a diagram $D$ and the dancers will dance the entire diagram. This means that one can label all strands of $D$ with virtual quandle elements by assigning labels at the $da^{(c)}(K)$ strands where the dancers are and then using the virtual quandle rules in \autoref{fig:virtquan} at each crossing to propagate the labels. The resulting propagation of labels may give rise to a virtual quandle coloring, or it may not. In conclusion, $Col_{|V|}(K)\leq |V|^{da^{(c)}(K)}$. The inequality is sharp when any propagation from the labels at the initial dancers becomes a virtual quandle coloring.
\end{proof}

A quandle can be thought of as a virtual quandle where $v$ is the identity map. So, a function of the number quandle colorings bounds $da^{(u)}$ from below. A function of the number virtual quandle colorings bounds $da^{(c)}$ from below by \autoref{lem:relatevirtquanwithdance}. The following theorem is a step in the right direction towards exhibiting the difference between $da^{(u)}$ and $da^{(c)}.$

\begin{thm}
    There exists an infinite family of virtual knots  $\{K_n\}_{n\in\mathbb{N}}$ such that $$\lim_{n\rightarrow \infty} \left[da^{(c)}(K_n)-\log_{|Q|} Col_Q(K_n)\right] = \infty$$ for any choice of finite quandle $Q$ (i.e., a virtual quandle with trivial automorphism).
\end{thm}

\begin{proof}
    Let $K_n$ be a connected sum of $n$ copies of the Kishino knot depicted in \autoref{fig:dancekish}. We see that $K_n$ is trivial as a welded knot. By a result in \cite{fenn1997braid}, if two knots $K$ and $J$ are welded equivalent, then $Col_Q(K)=Col_Q(J)$. It follows that $\log_{|Q|} Col_Q(K_n) = 1$ for any finite quandle $Q$. On the other hand, there is a virtual quandle $V$ of order 3 such that $Col_V(K)=9$ (see \cite[Example 37]{farinati2019virtual}). The readers can also check that one same strand in each of the nine colorings receives the label 1. We can perform a connected sum operation at this arc to create $K_n$ as shown on the right of \autoref{fig:dancekish}. Using induction, we can see that $Col_V(K_n) = 9^n.$
\end{proof}

Note that just because $\log_{|Q|} Col_Q(K_n) = 1$, this does not mean that $da^{(u)}(K_n) = 1$. Rather we mean to say that virtual knots with only trivial quandle colorings are good candidates to create large gaps between $da^{(u)}$ and $da^{(c)}$.

\begin{figure}
\labellist
\small\hair 2pt
\pinlabel $a$ at -9 133
\pinlabel $b$ at -7 30
\pinlabel $1$ at 23 80
\pinlabel $1$ at 87 80
\endlabellist
    \centering    \includegraphics[scale=.6]{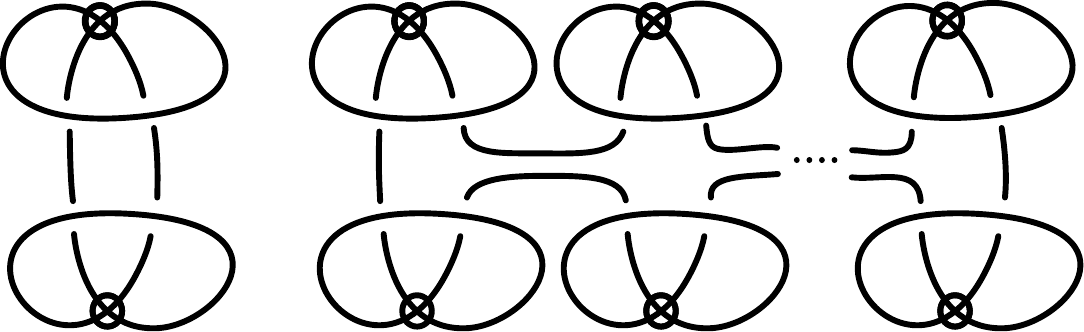}
    \caption{Left: the Kishino link $K$. Right: a connected sum of $n$ copies of Kishino links.}\label{fig:dancekish}
\end{figure}

To add one more invariant into the mix, there is a virtual analogue of Alexander's theorem which states every virtual link is the closure of a virtual braid \cite{Kauf} (see also \cite{KL}). 
As such, the ``virtual braid index'' (the minimum number of strands in a virtual braid presentation for the link), denoted $\beta(K)$, is a virtual link invariant. 
Just as in the classical setting, the $b_2$ bridge index is bounded above by the virtual braid index; the closure of an $n$-stranded braid is a Morse embedding of the virtual link with exactly $n$ maximums. 
Following essentially the same proof that the classical danceability index is bounded above by the braid index from \cite{ASS}, we can show that the coincident danceability index is bounded above by the virtual braid index.

\begin{prop}
    For all virtual links $K$, $da^{(c)}(K)\leq \beta(K)$.
\end{prop}

\begin{proof}
    Let $K$ be a virtual link with virtual braid index $n$, and let $D$ be a diagram for $K$ that realizes this braid index, that is, $D$ is the closure of an $n$-stranded virtual braid.
    Choose the orientation of $D$ induced by the upwards orientation of the braid.
 Place $n$ dancers, one on each strand, at the bottom of the braid. All dancers travel “up” the braid.
After the dancers traverse the braid, they travel the arcs in the closure, which connect back down to the initial starting points. Choose the speeds for each dancer so that all dancers travel up the braid uniformly with respect to height, not distance traveled. Just before each classical crossing, increase the speed of the dancer on the over-strand so that the dancer passes the crossing first before the dancer on the lower strand. After the crossing, decrease the over-strand dancer's speed to join in the uniform height progression of the other dancers. Notice that at each virtual crossing, the two dancers arrive at the crossing at the same time (the same height), satisfying the coincident rule. The resulting configuration is a coincident $n$-dance configuration for $D$, proving that $da^{(c)}(K)\leq da^{(c)}(D)\leq \beta(D)=\beta(K)$.
\end{proof}

Putting it all together, we see that for all virtual links $K$, the following inequalities are true (with currently unknown comparisons in braces). 

$$
da^{(u)}(K)=b_1(K)\leq \begin{Bmatrix}
b_2(K)\\
da^{(c)}(K)
\end{Bmatrix}\leq \beta(K)\text{\hspace{.5cm} and \hspace{.5cm}}
\begin{Bmatrix}
da^{(u)}(K)=b_1(K)\\
sw^{(v)}(K)\\
\log_{|Q|}Col_Q(K)
\end{Bmatrix}\leq da^{(c)}(K)
$$

\subsection{Python code}
In \cite{BKVV}, the authors provided Python code which calculates the Wirtinger number of a diagram. In this paper, we adapt their code to compute the danceability of a diagram. We have proven that danceability is equal to the bridge index. Thus, for readers interested in computing Wirtinger number, our code is not useful, because it will likely produce a higher estimate than needed. However, we successfully used it to find a diagram whose Wirtinger number is strictly less than its danceability number in \autoref{prop:differencebetweendanceandwirt}. Hence, the code can be used to demonstrate the difference between coloring diagrams with or without specified orientations.

One task for which our code is certainly useful is estimating the coincident danceability index. For this, we must determine how to encode a virtual knot diagram in a Gauss code \textit{with} virtual crossing data. Observe that the Gauss codes most often used to encode virtual knots do not keep track of virtual crossings. In our Python computations, we use non-integers to represent virtual crossings in Gauss codes. In the code, we impose the rule that at each virtual crossing, the coloring is not extended until both incoming strands are colored. The table of coincident numbers of virtual knots from Green’s table and our code can be found on GitHub at \cite{table}. 
In the table, the first column contains the names of knots used in Green’s table. The third column lists all strands. The fourth column indicates which strands can function as initial dancers. This process is demonstrated in \autoref{fig:499}.

\begin{figure}[h]
\labellist
\small\hair 2pt
\pinlabel $6.25$ at -20 215
\pinlabel $2$ [l] at 360 130
\pinlabel $1$ [t] at 218 28
\pinlabel $3$ [r] at 140 442
\pinlabel $4$ [bl] at 244 320
\pinlabel $3.25$ [bl] at 390 275
\endlabellist
    \centering    \includegraphics[scale=.25]{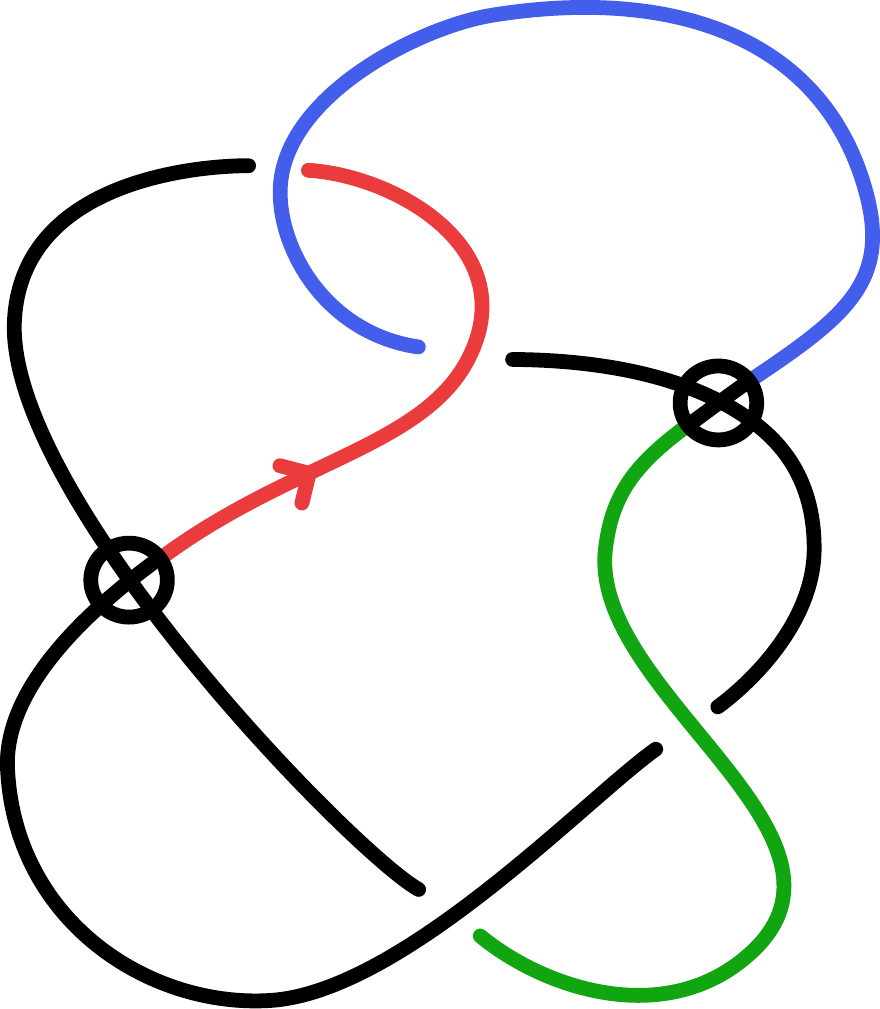}
    \caption{In the output table, the knot 4.99 has strands A: -1 2 -3.25 ; B: -3.25 3 -4 ; C: -4 3.25 ; D: 3.25 -2 ; E: -2 1 -6.25 ; F: -6.25 4 -3 ; G: -3 6.25 ; H: 6.25 -1. Three dancers start on the strands A, B, and F (colored)  and can dance the whole diagram. No pair of dancers can dance the entire diagram.}
\label{fig:499}
\end{figure}



\bibliography{DanceIsBridge.bib}

@article{ASS,
  author  = {Addison, Sol and Scherich, Nancy and Snodgrass, Lila},
  title   = "{D}anceability, {D}irected by {B}raid {I}ndex",
  journal = "Proceedings of Bridges 2024: Mathematics, Art, Music, Architecture, Education, Culture",
  year    = 2024,
  volume  = "",
  number  = "",
  pages   = "459--462"
}

@article{BK,
author = {Boden, Hans and Karimi, Homayun},
VOLUME = {28},
YEAR = {2022},
PAGES = {1372--1398},
title = {Classical results for alternating virtual links},
journal = {New York Journal of Mathematics},
doi = {10.48550/arXiv.2204.09767}
}

@article{BR,
  title={Minimal crossing number implies minimal supporting genus},
  author={Boden, Hans and Rushworth, William},
  journal={Bulletin of the London Mathematical Society},
  volume={53},
  number={4},
  pages={1174--1184},
  year={2021},
  publisher={Wiley Online Library}
}

@article{BKVV,
author = {Blair, Ryan and Kjuchukova, Alexandra and Velazquez, Roman and Villanueva, Paul},
title = {Wirtinger systems of generators of knot groups},
journal = {Communications in Analysis and Geometry},
  volume={28},
  number={2},
  pages={243--262},
  year={2020}
}

@article{farinati2019virtual,
    AUTHOR = {Farinati, Marco A. and Garc\'ia Galofre, Juliana},
     TITLE = {Virtual link and knot invariants from non-abelian {Y}ang-{B}axter 2-cocycle pairs},
  journal={Osaka Journal of Mathematics},
  volume={56},
  number={3},
  pages={525--547},
  year={2019}
}

@article{fenn1997braid,
  title={The braid-permutation group},
  author={Fenn, Roger and Rim{\'a}nyi, Rich{\'a}rd and Rourke, Colin},
  journal={Topology},
  volume={36},
  number={1},
  pages={123--135},
  year={1997},
  publisher={Elsevier}
}

@article{HKK,
author = {Hirasawa, Mikami and Kamada, Naoko and Kamada, Seiichi},
title = {BRIDGE PRESENTATIONS OF VIRTUAL KNOTS},
journal = {Journal of Knot Theory and Its Ramifications},
volume = {20},
number = {6},
pages = {881-893},
year = {2011},
doi = {10.1142/S0218216511009017},
URL = {  https://doi.org/10.1142/S0218216511009017},
eprint = { https://doi.org/10.1142/S0218216511009017}
}

@article{KL,
    AUTHOR = {Kauffman, Louis H. and Lambropoulou, Sofia},
     TITLE = {Virtual braids},
   JOURNAL = {Fundamenta Mathematicae},
    VOLUME = {184},
      YEAR = {2004},
     PAGES = {159--186}
}

@article{Kauf,
author = {Kauffman, Louis H.},
year = {1999},
pages = {663-691},
title = {Virtual Knot Theory},
volume = {20},
journal = {European Journal of Combinatorics},
doi = {10.1006/eujc.1999.0314}
}

@article{KPW,
  title={Biquandles, quivers and virtual bridge indices},
  author={Khandhawit, Tirasan and Pongtanapaisan, Puttipong and Wang, Brandon},
  journal={arXiv preprint arXiv:2504.10396},
  year={2025}
}

@article{KUR, title={Gauss paragraphs of classical links and a characterization of virtual link groups}, volume={145}, DOI={10.1017/S0305004108001151}, number={1}, journal={Mathematical Proceedings of the Cambridge Philosophical Society}, author={Kurlin, Vitaliy}, year={2008}, pages={129–140}}

@article {MIL,
  title={On the total curvature of knots},
  author={Milnor, John},
  journal={Annals of Mathematics},
  volume={52},
  number={2},
  pages={248--257},
  year={1950},
  publisher={JSTOR}
}

@article{NS,
author = {Nakanishi, Yasutaka and Satoh, Shin},
title = {Two definitions of the bridge index of a welded knot},
journal = {Topology and its Applications},
volume = {196},
pages = {846-851},
year = {2015},
doi = {10.1016/j.topol.2015.05.045},
URL = {  https://doi.org/10.1016/j.topol.2015.05.045},
eprint = { https://doi.org/10.1016/j.topol.2015.05.045}
}

@article{nelson2016quotient,
    AUTHOR = {Nelson, Sam and Tamagawa, Sherilyn},
      TITLE = {Quotient quandles and the fundamental {L}atin {A}lexander
              quandle},
  JOURNAL = {New York Journal of Mathematics},
    VOLUME = {22},
      YEAR = {2016},
     PAGES = {251--263}
}

@article{ceniceros2009virtual,
  title={{Virtual Yang-Baxter cocycle invariants}},
  author={Ceniceros, Jose and Nelson, Sam},
  journal={Transactions of the American Mathematical Society},
  volume={361},
  number={10},
  pages={5263--5283},
  year={2009}
}

@article{NPV,
  title   = {Biquandles, semi-Wirtinger number and bridge index},
  author  = {Nguyen, Thieu and Vo, Hanh and Pongtanapaisan, Puttipong},
  journal = {Communications of the Korean Mathematical Society},
  year    = {2026},
  volume  = {41},
  number  = {1},
  pages   = {247-272},
  doi     = {10.4134/CKMS.c240294},
  url     = {https://ckms.kms.or.kr/journal/view.html?doi=10.4134/CKMS.c240294},
}

@article {PP,
    AUTHOR = {Pongtanapaisan, Puttipong},
     TITLE = {Wirtinger numbers for virtual links},
   JOURNAL = {Journal of Knot Theory and its Ramifications},
    VOLUME = {28},
      YEAR = {2019},
    NUMBER = {14},
     PAGES = {1950086, 16},
      ISSN = {0218-2165,1793-6527},
   MRCLASS = {57K12},
  MRNUMBER = {4077437},
MRREVIEWER = {Margarita\ Maria\ Toro},
       DOI = {10.1142/S021821651950086X},
       URL = {https://doi.org/10.1142/S021821651950086X},
}

@article {R,
    AUTHOR = {Reidemeister, Kurt},
     TITLE = {Elementare {B}egr\"{u}ndung der {K}notentheorie},
   JOURNAL = {Abh. Math. Sem. Univ. Hamburg},
  FJOURNAL = {Abhandlungen aus dem Mathematischen Seminar der
              Universit\"{a}t Hamburg},
    VOLUME = {5},
      YEAR = {1927},
    NUMBER = {1},
     PAGES = {24--32},
}

@book{Rolfsen,
  title={Knots and Links},
  author={Rolfsen, Dale},
  isbn={9780914098164},
  lccn={lc76015514},
  series={Mathematics Lecture Series},
  url={https://books.google.com/books?id=qFLvAAAAMAAJ},
  year={1976},
  publisher={Publish or Perish}
}

@article{KS,
  author  = "Karl Schaffer",
  title   = "Dancing Topologically",
  journal = "Proceedings of Bridges 2021: Mathematics, Art, Music, Architecture, Education, Culture",
  year    = 2021,
  volume  = "",
  number  = "",
  pages   = "79--86"
}

@article{HS,
  title={{\"U}ber eine numerische Knoteninvariante},
  author={Schubert, Horst},
  journal={Mathematische Zeitschrift},
  volume={61},
  number={1},
  pages={245--288},
  year={1954},
  publisher={Springer}
}

@book{schultensbook,
    AUTHOR = {Schultens, Jennifer},
     TITLE = {{Introduction to 3-Manifolds}},
    SERIES = {Graduate Studies in Mathematics},
    VOLUME = {151},
 PUBLISHER = {American Mathematical Society, Providence, RI},
      YEAR = {2014},
     PAGES = {x+286}
}

@incollection{Schultenspaper,
  author      = {Schultens, Jennifer},
  title       = {The {B}ridge {N}umber of a {K}not},
  editor      = {Adams, Colin and Flapan, Erica and Henrich, Allison and Kauffman, Louis H. and Ludwig, Lewis D. and Nelson, Sam},
  booktitle   = {Encyclopedia of Knot Theory},
  publisher   = {Chapman and Hall/CRC},
  year        = {2021}
}

@misc{table,
  author = {Pongtanapaisan, Puttipong},
  title = {{Coincident dancebeality python code and table {G}it{H}ub repository, \url{https://github.com/nscherich/Coincident_danceability_code_and_table}}},
  year = 2026,
  url = {https://github.com/nscherich/Coincident_danceability_code_and_table},
  urldate = {}
}
\bibliographystyle{plain}

\end{document}